\documentclass[12pt]{article}
\usepackage{amsfonts,amsmath,amsthm,amssymb}

\def\hybrid{\topmargin 0pt      \oddsidemargin 0pt
        \headheight 0pt \headsep 0pt
        \voffset=-0.5cm
        \textwidth 6.5in        
        \textheight 9in         
        \marginparwidth 0.0in
        \parskip 5pt plus 1pt   \jot = 1.5ex}
\catcode`\@=11
\def\marginnote#1{}

\newcount\hour
\newcount\minute
\newtoks\amorpm
\hour=\time\divide\hour by60
\minute=\time{\multiply\hour by60 \global\advance\minute by-\hour}
\edef\standardtime{{\ifnum\hour<12 \global\amorpm={am}%
        \else\global\amorpm={pm}\advance\hour by-12 \fi
        \ifnum\hour=0 \hour=12 \fi
        \number\hour:\ifnum\minute<10 0\fi\number\minute\the\amorpm}}
\edef\militarytime{\number\hour:\ifnum\minute<10 0\fi\number\minute}

\def\draftlabel#1{{\@bsphack\if@filesw {\let\thepage\relax
   \xdef\@gtempa{\write\@auxout{\string
      \newlabel{#1}{{\@currentlabel}{\thepage}}}}}\@gtempa
   \if@nobreak \ifvmode\nobreak\fi\fi\fi\@esphack}
        \gdef\@eqnlabel{#1}}
\def\@eqnlabel{}
\def\@vacuum{}
\def\draftmarginnote#1{\marginpar{\raggedright\scriptsize\tt#1}}
\def\draftlabel#1{{\@bsphack\if@filesw {\let\thepage\relax
   \xdef\@gtempa{\write\@auxout{\string
      \newlabel{#1}{{\@currentlabel}{\thepage}}}}}\@gtempa
   \if@nobreak \ifvmode\nobreak\fi\fi\fi\@esphack}
        \gdef\@eqnlabel{#1}}
\def\@eqnlabel{}
\def\@vacuum{}
\def\draftmarginnote#1{\marginpar{\raggedright\scriptsize\tt#1}}

\def\draft{\oddsidemargin -.5truein
        \def\@oddfoot{\sl preliminary draft \hfil
        \rm\thepage\hfil\sl\today\quad\militarytime}
        \let\@evenfoot\@oddfoot \overfullrule 3pt
        \let\label=\draftlabel
        \let\marginnote=\draftmarginnote
   \def\@eqnnum{(\theequation)\rlap{\kern\marginparsep\tt\@eqnlabel}%
\global\let\@eqnlabel\@vacuum}  }

\def\numberbysection{\@addtoreset{equation}{section}
        \def\theequation{\thesection.\arabic{equation}}}

\def\underline#1{\relax\ifmmode\@@underline#1\else
        $\@@underline{\hbox{#1}}$\relax\fi}

\def\titlepage{\@restonecolfalse\if@twocolumn\@restonecoltrue\onecolumn
     \else \newpage \fi \thispagestyle{empty}\c@page\z@
        \def\thefootnote{\fnsymbol{footnote}} }

\def\endtitlepage{\if@restonecol\twocolumn \else  \fi
        \def\thefootnote{\arabic{footnote}}
        \setcounter{footnote}{0}}  
\relax

\numberbysection
\hybrid

\def\beq{\begin{equation}}
\def\eeq{\end{equation}}
\def\bea{\begin{eqnarray}}
\def\eea{\end{eqnarray}}
\def\p{\partial}
\def\G{\Gamma}

\def\e{\varepsilon}

\def\R{{\cal R}}
\def\A{{\cal A}}

\def\L{{\cal L}}

\def\M{{\cal M}}
\def\MH{{\cal M}_{g,n}(\underline{h})}
\def\TH{{\cal T}_{g,n}(\underline{h})}
\def\MR{{\cal M}^{real}_{g,n}(\underline{h})}
\def\oM{\overline{\cal M}}

\def\P{{\cal P}}

\def\dim{\operatorname{dim}}
\def\res{\operatorname{res}}
\def\Sym{\operatorname{Sym}}

\def\wt{\widetilde}
\def\wh{\widehat}
\def\un{\underline}
\def \matrix #1 {\left(\begin{array}{cc} #1 \end{array}\right)}

\newtheorem{theo}{Theorem}[section]
\newtheorem{prop}[theo]{Proposition}

\newtheorem{lem}[theo]{Lemma}
\newtheorem{conj}[theo]{Conjecture}

\theoremstyle{definition}
\newtheorem{df}[theo]{Definition}
\newtheorem{rem}[theo]{Remark}

\def\bbZ{{\mathbb Z}}
\def\bbR{{\mathbb R}}
\def\bbC{{\mathbb C}}
\def\bbP{{\mathbb P}}

\def\Mh{\M_{g,\wh 1}}

\begin{document}
\begin{titlepage}
\title{The universal Whitham hierarchy and geometry of the moduli  space of pointed
Riemann surfaces}
\author{S.Grushevsky
\thanks{Department of Mathematics, Princeton University,
Princeton, NJ 08544, USA; e-mail:
sam@math.princeton.edu. Research is supported in part by National
Science Foundation under the grant DMS-05-55867.}
\and I.Krichever
\thanks{Columbia University, New York, USA and Landau Institute for
Theoretical Physics, Moscow, Russia; e-mail:
krichev@math.columbia.edu.}}

\date{\today}

\maketitle

\begin{abstract}
We show that certain structures and constructions of the
Whitham theory, an essential part of the perturbation theory of soliton
equations, can be instrumental in understanding the
geometry of the moduli spaces of Riemann surfaces with marked points. We use the ideas of the Whitham theory to define local coordinates and construct foliations on the moduli spaces. We use these constructions to give a new proof of the Diaz' bound on the dimension of complete subvarieties of the moduli spaces. Geometrically, we study the properties of meromorphic differentials with real periods and their degenerations.
\end{abstract}

\end{titlepage}

\section{Introduction}
Solitons originally arose in the study of shallow waves. Since then, the
notion of soliton equations has broadened considerably, and it now embraces a wide class of
non-linear ordinary and partial differential equations, which all share the
characteristic feature of being expressible as a compatibility condition for an
auxiliary system of linear differential equations. The general algebro-geometric
construction of exact periodic and quasi-periodic solutions of soliton equations was proposed by the second-named author in
\cite{kr1,kr2}, where Baker-Akhiezer functions were introduced
(the analytical properties of Baker-Akhiezer functions are generalization of
properties of the Bloch solutions of the finite-gap Sturm-Liouville operators,
established during the initial development of the finite-gap integration
theory of the Korteweg-de Vries equation, see \cite{nov-dub-mat, mat-its, Mckean, Lax}).

The algebro-geometric solutions of soliton equations corresponding to smooth algebraic curves can be
explicitly written in terms of the Riemann $\theta$-function. The celebrated Novikov's
conjecture: the Jacobians of curves are exactly those indecomposable principally
polarized abelian varieties (ppav) whose theta-functions provide explicit solutions of
the Kadomtsev-Petviashvili (KP) equation, was the first evidence of the now well-accepted usefulness of combining the techniques of integrable systems and algebraic geometry to obtain new results in both fields.

Novikov's conjecture was proved by Shiota in \cite{shiota}, and until relatively
recently had remained the most effective solution of the Riemann-Schottky problem, the problem of characterizing Jacobians among all ppavs.
A much stronger characterization of Jacobians was suggested by Welters, who, inspired
by Novikov's conjecture and Gunning's theorem \cite{gun1}, conjectured in \cite{wel1}
that a ppav is a Jacobian if and only if its Kummer variety has at least {\it one} trisecant (and then it follows that in fact it has a four-dimensional family of trisecants).

Recall that for a ppav $X$ with principal polarization $\Theta$ the  Kummer variety $K(X)$ is the image of the complete linear system $|2\Theta|$. This is to say that the coordinates for the embedding $K:X/\pm 1\hookrightarrow\bbC\bbP^{2^g-1}$ are given by a basis of the sections of $|2\Theta|$, consisting of theta functions of the second order
$$
 \Theta[\e](z):=\theta[\e,0](2\tau,2z):=\sum\limits_{n\in\bbZ^g} \exp (\pi i (2n+\e)^t\tau(n+\e/2)+4(n+\e/2)^t z)
$$
for all $\e\in(\bbZ/2\bbZ)^{2g}$, where $\tau$ is the period matrix of $X$. A projective $(m-2)$-dimensional plane
$\mathbb{CP}^{m-2}\subset\mathbb{CP}^{2^g-1}$ intersecting $K(X)$ in at least $m$ points is called
an $m$-secant of the Kummer variety.

The Kummer varieties of Jacobians of curves were shown to admit a four-dimensional family of trisecant lines (this is the Fay-Gunning trisecant formula, see
\cite{fay}). It was then shown by Gunning \cite{gun1} that the existence of a
one-dimensional family of trisecants that are translates of each other characterizes Jacobians among
all ppavs. The Welters' trisecant conjecture was recently proved by the second-named
author in (\cite{kr-schot,kr-tri}). In \cite{kr-prym, kr-quad} the soliton
theory was used to obtain a solution of another classical problem of algebraic geometry ---
characterizing Prym varieties among indecomposable ppavs.

\smallskip
The algebro-geometric perturbation theory for two-dimensional soliton equations
was developed in \cite{kr-av,kr-tau}. It was stimulated by the application of the
Whitham approach for (1+1) integrable equation of the KdV type, see \cite{gur,ffm,dob}. As
usual, in the perturbation theory ``integrals'' of an initial equation become functions
of the ``slow'' variables $\e t_A$ (where $\e$ is a small parameter). ``The
Whitham equations'' is a name given to equations that describe ``slow'' variations of
``adiabatic integrals''.

\smallskip
We denote by $\M_{g,n}$ the moduli space of smooth algebraic
curves $\G$ of genus $g$ with $n$ distinct labeled marked points $p_1,\ldots, p_n$ (i.e.~not taking the quotient under the symmetric group). The universal Whitham hierarchy, as defined in \cite{kr-tau} is
a hierarchy of commuting equations on the total space of the bundle $\wh \M_{g,n}\to\M_{g,n}$ of infinite rank: $\wh\M_{g,n}$ is the moduli space of  smooth
algebraic curves $\G$ of genus
$g$ with $n$ labeled marked points and a choice of a holomorphic local coordinate $z_i$, which we think of as an infinite power series, in a neighborhood $U_{p_i}$ of each marked point $p_i$ (it is customary in the theory of integrable equations to write the local coordinate as $k^{-1}$ instead of $z$, so that we have $k^{-1}_i(p_i)=0$, i.e.~$k_i(p_i)=\infty$). Thus we have the bundle
\begin{equation}
\wh \M_{g,n} = \{\G,p_i\in\G,z_i:U_{p_\alpha}\to\bbC,
\ i=1,\ldots,n\} \quad\longrightarrow\quad\M_{g,n}=\{\G,p_i\in\G\rbrace.
\label{Mgn}
\end{equation}
For any point in $\wh\M_{g,n}$ and for a point of the Jacobian $J(\G)$
the algebro-geometric construction gives a quasi-periodic solution of some
integrable partial non-linear differential equation (for a given non-linear
integrable equation the corresponding set of data has to be specified. For
example, the solutions of the KP hierarchy correspond to the
case $n=1$. The solutions of the two-dimensional Toda lattice correspond to the case
$n=2$).

The construction of special ``algebraic'' orbits of the Whitham hierarchy, proposed in
\cite{kr-av}, has already found its applications to the theory of topological quantum
field models and to Seiberg-Witten solution of $N=2$ supersymmetric gauge models (see
details in \cite{kp1,kp2} and references therein).

The moduli spaces of curves with marked points have curious vanishing
properties of tautological classes, Chow groups and rational cohomology. In \cite{arb}
Arbarello constructed a stratification of the moduli space $\M_g$ of smooth Riemann
surfaces of genus $g$ and provided some evidence that it can be a useful tool for
investigating the geometric properties of $\M_g$. Later Diaz in \cite{diaz} used a
variant of Arbarello stratification to show that $\M_g$ does not contain complete
(complex) subvarieties of dimension $g-1$. Some years later, using a similar
stratification Looijenga proved in \cite{loo} that the tautological classes of degree greater than
$g-2$ vanish in the Chow ring of $\M_g$, which implies Diaz' result (the Hodge class $\lambda_1$ is ample on $\M_g$, and thus for any complete $d$-dimensional subvariety $X\subset\M_g$ we would have $\lambda_1^d\cdot X>0$, while $\lambda_1^{g-1}=0$, as a tautological class).
Hain and Looijenga \cite{halo} then asked whether the reason for this vanishing would be the existence of a cover of the moduli space of curves by at most $g-1$ affine open sets. Roth and Vakil \cite{rova} studied affine stratifications and asked whether they could be given for $\M_g$. The following conjecture is widely believed to be true, and would imply all of the above results
\begin{conj}[\cite{loo,halo,rova}] Let $g,n\geq 0$ be such that $2g-2+n>0$.
Then the moduli space $\M_{g,n}$
of smooth genus $g$ algebraic curves with $n$ marked points
has a stratification
\begin{equation}\label{looij}
\M_{g,n}=\bigcup_{i=1}^{g-\delta_{n,0}} S_i,\ \ \bar S_j=\bigcup_{i\leq j} S_i
\end{equation}
such that each locally closed stratum $S_i$ is affine.
\end{conj}
Existence of such a stratification would also imply bounds on the homotopical dimension of $\M_{g,n}$ obtained by Mondello \cite{mo}.

In \cite{fo-la} Fontanari and Looijenga construct such affine  stratifications for
$\M_g$ for $g\le 5$. Both Arbarello's and Diaz' stratifications of $\M_g$ have the
right number of strata, but it is not known whether the strata are affine. Not even conjectural candidates have been proposed for covers of $\M_g$ by at most $g-1$ affine open sets.

Our first goal is to give a new proof of Diaz' theorem. Our approach, based on the
constructions of the Whitham theory, does not use any kind of stratification. What we do in a sense can be considered a generalization of the constructions and computations in the Hurwitz space (the space of Riemann surfaces together with a meromorphic function with prescribed pole orders), which were used to great effect in \cite{elsv}). Our construction is on the total space of the bundle of meromorphic differentials on Riemann surfaces with prescribed poles. If the meromorphic differential has no residues at marked points, it is the derivative of a meromorphic function, and thus the Hurwitz space is a subvariety of this space of differentials corresponding to the case when the differential is exact, i.e.~when all its periods are equal to zero. There are two advantages to our construction: that the compactification is straightforward, and that the subbundle where the singular parts are fixed admits a section, the unique meromorphic differential with prescribed singular parts and real periods (if all these real periods are in fact zero, we again recover the Hurwitz space).

We use this construction of {\it real-normalized} differentials to show that on the moduli spaces of curves with fixed finite jets of local
coordinates at marked points there exist canonical {\it real-analytic} local coordinates. Moreover, part of these local coordinates are in fact globally defined real functions, which become harmonic when
restricted to leaves of a canonically defined foliation on the moduli space. The
maximum principle for harmonic functions then implies that codimension of any compact
cycle in the moduli space can not be less than  the dimension of the corresponding
foliation, and we thus obtain a proof of Diaz' theorem.

\section{Algebraic orbits of the Whitham hierarchy}
The notion of {\it algebraic orbits} of the universal Whitham hierarchy is at the heart
of all the following constructions. They are defined as leaves of a certain ``canonical''
foliation on the moduli space of curves with marked points, together with a meromorphic differential with prescribed pole orders at the points.

\begin{df}
For any set of positive integers $\un{h}=h_1,\ldots, h_n$ we denote by
$$
 \MH:=\left\{(\G,p_1,\ldots,p_n)\in\M_{g,n};
 \omega\in H^0(K_\G+\sum h_i p_i)\setminus\mathop{\cup}\limits_{j=1}^n H^0(K_\G-p_j+\sum h_i p_i)\right\}
$$
the moduli space of curves of genus $g$ with $n$ marked points, together with a meromorphic differential with poles of order {\it exactly} $h_i$ at each $p_i$ (in the literature on integrable systems $\omega$ is often denoted $dE$ where $E$ is thought of as the energy functional. It does not mean that $dE$ is then assumed to be exact).
The residues of the differential give $n$ global well-defined functions on $\MH$:
\begin{equation}\label{rescoords}
 \rho_i(\G,\un{p},\omega):=\res_{p_i}\omega.
\end{equation}
\end{df}

\begin{rem}
Note that in the definition we have required the poles to be of order exactly $h_i$ --- and thus obtained an open subset of the moduli of curves with differentials of poles of order at most $h_i$. This will be useful for our construction of local coordinates and foliations in $\MH$, as we will be able to say that the degree of any differential in $\MH$ is the same, and thus it has a fixed number of zeros.
\end{rem}

The moduli space $\MH$ is an open subset of the total space of the complex vector bundle $K_\Gamma+\sum h_ip_i$ of meromorphic differentials over $\M_{g,n}$, and thus is a complex orbifold of complex dimension
\begin{equation}\label{dimMgnh}
 \dim_\bbC\MH=\dim_\bbC\M_{g,n}+h^0(K_\G+\sum\limits_{i=1}^n h_i p_i)=3g-3+n+g-1+\sum\limits_{i=1}^n h_i
\end{equation}
(we have $h^0(K_\G)=g$, and for a meromorphic differential the only condition is for the sum of the residues to be zero).

The moduli space of curves with marked points and an {\it exact} meromorphic differential with prescribed pole orders is a subset of $\MH$. Thus the Hurwitz space --- the moduli space of curves with marked points and a meromorphic function with pole orders $h_i-1$ at marked points --- is the subset of $\MH$ consisting of differentials whose integral over any closed curve is zero. This is to say these are differentials with all periods zero, and all residues zero. We will now describe a canonical foliation on $\MH$, for which this Hurwitz space will one leaf. To define this foliation, we will perform a construction on the universal cover, (an open set in) the total space of a vector bundle over the Teichm\"uller space, and then argue that the construction is invariant under the action of the mapping class group.
\begin{df}
For a curve $\G\in\M_g$ we call a basis $A_1,\ldots,A_g,$ $B_1,\ldots,B_g\in H_1(\G,\bbZ)$ standard if the intersection numbers are $A_i\cdot A_j=B_i\cdot B_j=0$, $A_i\cdot B_j=\delta_{ij}$. We denote $\TH$ the moduli of objects as in $\MH$ together with a choice of a standard basis. On $\TH$ we have the well-defined global functions
\begin{equation}\label{intcoords}
  \alpha_i(\G,\un{p},\omega,\un{A},\un{B}):=\oint_{A_i}\omega; \qquad
  \beta_i(\G,\un{p},\omega,\un{A},\un{B}):=\oint_{B_i}\omega
\end{equation}
(from the point of view of integrable systems, these integrals are some of the times for the universal Whitham hierarchy, and thus often denoted $T$). We then define the foliation $\L$ on $\TH$ to have leaves given by, for any complex numbers $r_1,\ldots,r_n,a_1,\ldots,a_g,b_1,\ldots,b_g$,
\begin{equation}\label{leaves}
  L_{\un{r},\un{a},\un{b}}:=\left\{(\G,\un{p},\omega,\un{A},\un{B})\in\TH \mid \rho_j=r_j, \alpha_i=a_i, \beta_i=b_i,\ \forall j=1,\ldots,n,\ \forall i=1,\ldots, g\right\}
\end{equation}
\end{df}
Notice that on $\MH$ one cannot globally talk about periods, as there is no chosen basis of cycles, and thus there are no global functions $\alpha_i$ or $\beta_i$. However, the condition of the periods being constant is independent of the choice of the basis, and we thus have
\begin{lem}
The subvarieties $L_{\un{r},\un{a},\un{b}}\subset\TH$ are permuted by the action of the mapping class group, and thus the family $\L$ of them for all values of $r,a,b$ defines a complex foliation of $\MH$ by complex submanifolds. By foliation here we only mean that there exists a leaf through every point of $\MH$, and that no two leaves intersect.
\end{lem}
\begin{proof}
Indeed, note that though the definition of $L_{\un{r},\un{a},\un{b}}$ depends on the choice of the basis $\un{A},\un{B}$ for $H_1(X,\bbZ)$ and thus only makes sense on $\TH$, if we choose a different basis $\un{A'},\un{B'}$, then the new basis is obtained from the old one by a linear transformation $G\in Sp(2g,\bbZ)$, and thus the periods $\alpha',\beta'$ of a differential with respect to the new basis are obtained by applying $G$ to the periods $\alpha,\beta$, and thus the manifold $L_{\un{r},\un{a},\un{b}}$ is mapped to $L_{\un{r},G(\un{a},\un{b})}$, so the action of the mapping class group permutes the leaves, and thus preserves the foliation.
\end{proof}
Each leaf of $\L$ is given by $n-1+2g$ equations (recall that the sum of the residues is equal to 1), and we thus expect it to have codimension $n-1+2g$ in $\MH$. To prove that the codimension of a leaf is indeed this, one needs to show that the functions $\rho,\alpha,\beta$ are independent, i.e.~that prescribing their values imposes independent conditions on $\TH$. It turns out that this is indeed the case, and moreover that this set of functions can be completed to a local coordinate system near any point of $\MH$ (for local coordinates it does not matter whether we work with $\MH$ or the universal cover $\TH$). This will also imply that the leaves of $\L$ are smooth. The construction of such local coordinates is given in \cite{kp1}. We will now summarize it for completeness and for future use.

\begin{rem}
For motivation, note that the leaf $L_{\un{0},\un{0},\un{0}}$ corresponding to zero values of all periods and residues, consists of exact differentials, and thus is simply the Hurwitz space of meromorphic functions with prescribed pole orders at the marked point. As described by Ekedahl, Lando, Shapiro, and Vainshtein in \cite{elsv}, the coordinates  on the Hurwitz space are given by the Lyashko-Looijenga mapping: associating to the meromorphic function the (unordered) set of its critical values.

The critical point of a function is a zero of its differential, and so makes sense in our situation. The critical value of a function at the critical point is the integral of the differential --- and thus we need to fix the path of integration. Note that when dealing with Hurwitz spaces, one often allows only simple branching away from infinity, i.e.~requires all critical values to be distinct. We will not require this, and thus to get real analytic coordinates also along the locus where some critical values are multiple, we need to use the symmetric functions of critical values as coordinates rather than the critical values themselves.

Analogously to the Hurwitz space situation, where functions are only defined up to an additive constant, our construction should rather be performed on the moduli of curves with marked points and a chosen (multivalued) abelian integral, i.e.~a chosen integral of the meromorphic differential --- which is unique up to an additive constant. This space is an affine bundle over the moduli of curves with marked points, and is of independent interest, but we will not need the details of it for what follows.
\end{rem}

To formally define local coordinates, we will use the critical values of the integral of $\omega$. Indeed, write the divisor of $\omega$ on $\G$ as
\begin{equation}
 (\omega)=\sum_{s=1}^{2g-2+\sum h_i} q_s
\end{equation}
(where some of the $q_s$ may be the same). Consider then the integrals
\begin{equation}\label{defphi}
 \phi_j:=\int_{q_1}^{q_j}\omega,\qquad{\rm for}\ j=2,\ldots,2g-2+\sum h_i
\end{equation}
where the integral of course depends upon the path of integration, and on the choice of the numbering of the points in the divisor of $\omega$.  Note that even if we work on $\TH$ and require the path  not to intersect any of the loops $A_i$ or $B_i$, the integral still depends on the path of integration if the residues of $\omega$ are not all zero. We let
\begin{equation}\label{defsigma}
 \sigma_k:=s_k(\phi_2,\ldots,\phi_{2g-2+\sum h_i})\qquad{\rm for}
 \ s=1\ldots 2g-3+\sum h_i
\end{equation}
be the values of the elementary symmetric polynomials of the critical values $\phi_j$.

Locally in a neighborhood of a point $(\G,\un{p},\omega)\in\MH$ we can choose a basis for cycles, a labeling for the points in the divisor, and a family of paths of integration. Notice that even if some of the points $q_s$ coincide, the tuple $\un{q}\in  \Sym^{2g-2+\sum h_i}(\G)$ deforms holomorphically, and thus we can locally choose the labeling of the points, so that each point varies holomorphically. Thus in a neighborhood of any point of $\MH$ we can choose local holomorphic functions $\phi_j$, dependent on the choices made.
\begin{theo}[\cite{kp1}, Appendix]\label{coords}
The set of functions $\alpha,\beta,\sigma(\phi)$ (note that the total number of functions is $n-1+2g+2g-3+\sum h_i=\dim_\bbC\MH$)
give local holomorphic coordinates in a neighborhood of any point of $\MH$, dependent on the choices made above for defining $\alpha,\beta,\sigma(\phi)$.
\end{theo}
\begin{proof}
We outline the key step in the proof of the argument, given in full detail in \cite{kp1}. Suppose that the differentials of these functions are linearly dependent at some point $\G_0\in\MH$ (and thus the functions do not give local coordinates near $\G_0$). Then there exists a one-dimensional family $\G_t\subset\MH$, with complex parameter $t$ such that the derivative of any of the above functions along this family is equal to zero at $t=0$.

Choose locally a basis $\un{A},\un{B}$ for cycles, and let $\Omega_i(t)$ be the basis of holomorphic differentials on $\G_t$ dual to $A_i$. Denote then $F_i(p,t):=\int_{q_1(t)}^p\Omega_i(t)$ the corresponding abelian integral --- the function of $p\in\G_t$, depending on the choice of the path of integration. Let us also denote $f_t(p):=\int_{q_1(t)}^p\omega(t)$ the integral of our chosen meromorphic differential along the same path.

We will now want to see how $F_i$ varies in $t$. For this to make sense as a partial derivative, we need to ``fix'' the point $p$ as we vary $\G_t$, and to do this we will use $f$ as the local coordinate on the universal cover of $\G$. This is to say that we will fix $x:=f_t(p)$ and let $t$ vary; this is to say that $f$ allows us to define a connection on the space of abelian integrals.

Rigorously, we consider the derivative
\begin{equation}\label{partial}
 \frac{\p}{\p t}F_i(f_t^{-1}(x),t)|_{t=0}
\end{equation}
and show that it is zero. We think of the surface $\G_t$ as cut along a basis of cycles, so that the integrals $\int_{q_1(t)}^p$ in the definition above are taken along paths not intersecting this basis, i.e.~on the simply-connected cut surface. Then the expression above is by definition a differential on the cut surface $\G_0$, with possible poles at the zeros of $\omega$ (where $f^{-1}$ is singular), and with discontinuities along the cuts. However, if as we wary $t$ the coordinates $\alpha$ and $\beta$ do not change, i.e.~the periods of $\omega_t$ do not change, (\ref{partial}) has no discontinuity along the cut (the discontinuity of $F_i$ does not depend on $t$), and since the critical values $\phi$ do not change (which is implied by $\sigma(\phi)$ not changing as we vary $t$), (\ref{partial}) also has no poles at the critical values of $\omega$, as the singular part of $F_i(f_t^{-1}(x),t)$ there also does not depend on $t$. Thus expression (\ref{partial}) is an abelian differential on $\Gamma_0$ with zero $A$-periods; thus it is identically zero, and also has zero $B$-periods. Since the  $B$-periods of $\Omega_i$ are entries of the period matrix $\tau$ of $\G_0$, this means that we have
$$
 \frac{\p}{\p t}\tau_{ij} (t)|_{t=0}=0,
$$
The infinitesimal Torelli theorem says that the period map $\tau:\M_g\to\A_g$ induces an embedding on the tangent space away from the locus of the hyperelliptic curves (where the kernel of $d\tau$ is one-dimensional at the hyperelliptic curves), and thus the above is impossible unless $\G_0$ is a hyperelliptic curve. Further arguments complete the proof of the claim in the hyperelliptic case (one uses the explicit form of the degeneracy of the infinitesimal Torelli), we refer to the appendix of \cite{kp1} for details.
\end{proof}

\section{Differentials with real periods}
We now introduce the second main tool of the theory: differentials with real periods, or {\it real-normalized} abelian differentials, i.e.~differentials $\omega$ on a Riemann surface $\G$ such that all their periods are real.

\begin{df}
We denote by $\MR$ the space of curves with marked points, together with a real-normalized meromorphic differential with prescribed pole orders:
\begin{equation}\label{MrH}
 \MR:=\left\{(\G,\un{p},\omega)\in\MH\, \Big|\, \oint_\gamma \omega\in\bbR,\ \forall \gamma\in H_1(\G\setminus\{p_1,\ldots, p_n\},\bbZ)\right\}.
\end{equation}
Note that to be able to talk of all periods of a meromorphic differential being real, without choosing a basis for cycles, the integrals of $\omega$ around all poles need to be real, so all residues need to be imaginary. Thus we have $\rho_j:\MR\to i\bbR$.
\end{df}
Notice that the condition of the periods being real is a real-analytic and not a holomorphic condition. From now on our constructions will happen in the real-analytic category unless stated otherwise.
\begin{rem}
Because the periods of differentials are constant along the leaves of the foliation $\L$ on $\MH$, the foliation $\L$ restricts to a foliation on $\MR$ (i.e.~any leaf of $\L$ on $\MH$ intersecting $\MR$ is contained in $\MR$). Since $\MR$ is only a real-analytic orbifold, the foliation $\L$ on it is real-analytic, but each individual leaf carries the structure of a complex orbifold (recall that the smoothness of every leaf follows from theorem \ref{coords}). The functions $\sigma(\phi)$ give local holomorphic coordinates on the leaves of $\L$ on $\MR$.
\end{rem}

One of the strengths of the real normalization lies in the following
\begin{prop}\label{fglob}
When restricted to $\MR$, the tuple of imaginary parts of the functions $\phi$ given by (\ref{defphi}) can be defined globally: we can define global real-analytic maps
$$
  s_j:=s_j({\rm Im}\, \phi_2, \ldots, {\rm Im}\,\phi_{2g-2+\sum h_i}):\MR\to \bbR_{\ge 0}\qquad {\rm for}\ j=1\ldots 2g-3+\sum h_i.
$$
by taking the elementary symmetric polynomials of the imaginary parts ${\rm Im}\,\phi_j$ (notice that this is of course not the same as  ${\rm Im}\,\sigma_j$ defined by (\ref{defsigma})).
\end{prop}
\begin{proof}
For a real-normalized differential $\omega$, for any point $p_0\in\G\setminus\{\un{p},\un{q}\}$ (where $q$ are the zeroes of $\omega$, as above) the imaginary part of the integral $\int_{p_0}^{q_j}\omega$ is independent of the choice of the path of integration --- going around any cycle adds a real number to the integral. We now choose $q_1$ to be a zero of $\omega$ for which $\int_{p_0}^{q_1}\omega$ is minimal, and get a well-defined collection of the non-negative imaginary parts of critical values (at $q_i$ minus at $q_1$), of which we then take the elementary symmetric functions.
\end{proof}

The power of the real-normalization is in the uniqueness of a real-normalized differential with prescribed singular parts at the marked points (written out in terms of the jets of local coordinates at the marked points). Thus the real-normalized differentials provide a section of the bundle of meromorphic differentials with prescribed pole orders over the moduli space of curves with marked points, endowed with jets of local coordinates at these points.

\begin{prop}\label{Psi}
For any $(\G,p_1,\ldots,p_n)\in\M_{g,n}$, any set of positive integers $h_1,\ldots,h_n$, and any choice of $h_i$-jets of local coordinates $z_i$ in the neighborhood of marked points $p_i$, with $z_i(p_i)$ and any singular parts (i.e.~for $i=1\ldots n$ the choice of Taylor coefficients $c_i^1,\ldots, c_i^{h_i}$, with $\sum c_i^1=0$) there exists a unique real-normalized differential $\Psi$ on $\Gamma$ with prescribed singular parts, i.e.~such that in a neighborhood $U_i$ of each $p_i$ we have
$$
 \Psi|_{U_i}=\sum\limits_{j=1}^{h_i} c_i^j\frac{dz}{z^j} +O(1)
$$
(note that in the integrable systems literature $\Psi$ is often denoted $dE$ and coordinate $k_i=z_i^{-1}$ with $k_i(p_i)=\infty$ is used).
\end{prop}
\begin{proof}
Indeed suppose there were two such differentials. Subtracting one from the other would then yield a holomorphic differential $\Omega$ with all periods real. Then all the periods of the difference $\Omega-\bar{\Omega}$ must be zero, and thus we must have $\Omega-\bar\Omega=0\in H^1(X)$. Since $\Omega$ is holomorphic and $\bar\Omega$ is antiholomorphic, this implies $\Omega=\bar\Omega=0$.

From Riemann-Roch theorem it follows that there must exist a differential $\omega$ with the prescribed singular part. Let $a_1,\ldots,a_g$ be its periods over the $A$ cycles, and let $\Omega_1\ldots\Omega_g$ be the basis of holomorphic differentials dual to the $A$ cycles. The differential $\omega':=\omega-\sum a_i\Omega_i$ then has all $A$-periods zero (and thus in particular real). We now need to show that there exists a differential $\Psi=\omega'-\sum c_i\Omega_i$, for some $c_i\in\R$, with all $B$-periods real --- its $A$-periods are equal to $c_i$. Indeed, let $b_1,\ldots,b_g$ be the imaginary parts of the $B$-periods of $\omega'$. Since the imaginary part of the period matrix $\tau$ of $X$ is non-degenerate, there must exist a vector $c\in\R^g$ such that $b=({\rm Im}\,\tau)c$, and this is our solution.
\end{proof}

\section{A foliation of $\M_{g,2}$, and Diaz' theorem}
In what follows we will concentrate on meromorphic differentials with a single double pole (and thus with no residue at the marked point) --- traditionally called second kind --- i.e.~sections of $K_X+2p$ over $\M_{g,1}$, and meromorphic differentials with two simple poles with opposite residues --- traditionally called third kind --- sections of $K_X+p_1+p_2$ over $\M_{g,2}$.

\begin{df}
For the case of a differential of the second kind the singular part is equal to $r z^{-2} dz$ for some $r\in\bbC$, where $z$ is the local coordinate near $p$ with $z(p)=0$. If the local coordinate $z$ is changed, $r$ transforms as a tangent vector, and thus a Riemann surface with a differential of a second kind determines a point in the moduli space $\M_{g,1}$ together with a tangent vector at $p$. We will denote by $\Mh$ this space: the total space of the universal tangent bundle to $\M_{g,1}$ at the marked point.
\end{df}
\begin{rem}
Notice that such a differential of the second kind has no residue, and thus we can talk about real normalization. If such a differential of the second kind were exact, it would be the derivative of a meromorphic function with a single simple pole, in which case the Riemann surface would be $\bbC\bbP^1$. In \cite{elsv} it is explained why the tangent vectors at marked points appear in the context of Hurwitz spaces.
\end{rem}

For a differential of the third kind the singular part is determined by the residue $r\in\bbC$ at $p_1$ (the residue at $p_2$ is then $-r$). To be able to talk about real normalization, we need to require this residue to be imaginary, and then all such real-normalized differentials of the third kind are $\bbR$-multiples of each other. We fix the residue to be $i$ then (getting a section of the $\bbR$-line bundle $\M_{g,2}^real(1,1)\to\M_{g,2}$), and thus the foliation $\L$, along the leaves of which the residue is constant, induces a real-analytic foliation on $\M_{g,2}$. Each leaf of this foliation itself carries a complex structure, compatible with the complex structure on $\M_{g,2}$, and is of complex codimension $g$ in $\M_{g,2}$.

\begin{rem}
Real-normalized differentials of the third kind
probably had been known to Riemann in his study of electric potential created by two
charged particle on a surface. In modern literature they were used in \cite{wolpert}
to study triangulations of moduli space of curves with marked points in connection with light-cone string theory,
and in \cite{fourie} in a construction of an analog of Fourier-Laurent theory on Riemann
surfaces.
\end{rem}

The fiber of the forgetful map $\M_{g,2}\to\M_g$ over the point $[\G]$ is $\G\times\G\setminus{\rm diagonal}$, and thus non-compact. We define a partial compactification $\wt{\M}_{g,2}$ of $\M_{g,2}$ by allowing the two marked points to collide, so that the fiber of the map $\wt{\M}_{g,2}\to\M_g$, equal to $\G\times\G$, is compact. From the point of view of the Deligne-Mumford compactification, if the two marked points coincide, we attach a nodal $\mathbb{CP}^1$ at this point.

In the next section we formally study the degenerations of the real-normalized differentials of the second kind. For the case of the differential of the third kind the situation with this degeneration is clear: the bundle $K_\G+p_1+p_2$ extends to the boundary of $\overline{\M_{g,n}}$ as a line bundle $\omega_\G+p_1+p_2$, where $\omega_\G$ is the relative dualizing sheaf of a stable curve. Thus the limit of $\Psi$ is a differential with possible poles at the node and at the marked points on $\bbC\bbP^1$. However, as the node on the original curve is separating, and is the only marked point on the genus $g$ component of the stable curve, there can be no residue at it. Thus in this limit the real-normalized differential $\Psi$ becomes holomorphic on $\G$, and will have residues at the two other marked points of the $\mathbb{CP}^1$.
Since in the limit $\Psi$ is still real-normalized, in the limit ot becomes identically zero on $\Gamma$, and the associated functions $\sigma,\alpha,\beta$ all become zero (they no longer give coordinates near such degenerate points).

We now use this foliation on $\wt{M}_{g,2}$ and the local coordinates on it to prove Diaz' theorem.
\begin{theo}[Diaz \cite{diaz}]
There do not exist complete complex subvarieties of $\M_g$ of complex dimension greater than $g-2$.
\end{theo}
\begin{proof}
Suppose for contradiction $Y\subset\M_g$ were a complete complex submanifold of $\M_g$ with $\dim_\bbC Y\ge g-1$. Consider the preimage $Z\subset\wt{\M}_{g,2}$ of $Y$; it would then be a complete complex submanifold with $\dim_\bbC Z\ge g+1$ (notice that here we need the fact that the fibers of $\wt{\M}_{g,2}$ are complete, which is why we could not use $\M_{g,2}$ in the first place).

By proposition \ref{fglob} we have the globally defined real-analytic functions $s_j:\wt{\M}_{g,2}\to \bbR_{\ge 0}$. Let us now arrange the imaginary parts of the critical values ${\rm Im}\,\phi_j$ into functions $f_1\ge\ldots \ge f_{2g-1}\ge 0$, so that each $f_i$ is  piecewise real analytic (they may not be smooth where two of them coincide) and continuous. Then $f_1$ must achieve a maximum on $Z$, as a continuous function on a compact set. We will want to use local coordinates on $\M_{g,2}$ given by theorem \ref{coords}, and thus need to avoid working on $\wt{\M}_{g,2}\setminus\M_{g,2}$, where we do not have local coordinates. We start by proving the following
\begin{lem}
The maximum of $f_1$ on $Z$ is strictly greater then zero (and thus is achieved on $\M_{g,2}$).
\end{lem}
\begin{proof}
Indeed, if the maximum of $f_1$ were zero, $f_1$, and thus all $f_i$ and all $s_j$ would be identically zero on $Z$. Take then any point $(\Gamma,p_1\ne p_2)\in Z\cap\M_{g,2}$ (this is non-empty, as $Z$ is a preimage of $Y\subset\M_g$) and consider a leaf $L$ containing it. Then in a neighborhood of $(\G,p_1,p_2)$ all functions ${\rm Im}\, \phi_j$ would be identically zero along $L\cap Z$. Since these are imaginary parts of local holomorphic functions, the local holomorphic functions $\phi_j$ would be constant along $L\cap Z$, while all $\alpha$ and $\beta$ are constant on $L$. Thus the values of $\sigma(\phi),\alpha,\beta$ would all be constant along $L\cap Z$ locally near $(\Gamma,p_1,p_2)$. Since by theorem \ref{coords} these functions are local coordinates on $\M_{g,2}$ , this would imply that $(\Gamma,p_1,p_2)$ is an isolated point of $L\cap Z$. However, since $\operatorname{codim}_\bbC L=g$, and we assumed $\dim_\bbC Z\ge g+1$, we have $\dim_\bbC (L\cap Z)\ge 1$, and thus there is a contradiction.
\end{proof}

The proof of the theorem follows the same line of thought, but an inductive argument is needed. Indeed, we let  $Z_1$ be the locus of points in $Z$ where $f_1$ achieves its maximum. Since this maximum is non-zero, $Z_1$ is a closed subvariety of $Z\cap\M_{g,2}$ (and thus we know that we have local coordinates at any point of $Z_1$). We claim that $Z_1$ is foliated by (the connected components of) leaves of $\L|_Z$, i.e.~that if any component $(L\cap Z)^0$ of $L\cap Z$ contains a point of $Z_1$, then $(L\cap Z)^0\subset (L\cap Z_1)$. Note that $Z_1\subset Z$ is by definition closed and thus compact.

Indeed, take some $(\Gamma,p_1,p_2)\in Z_1$ and consider the leaf $L$ containing $(\Gamma,p_1,p_2)$. Since the complex codimension of $L$ is equal to $g$, and by assumption the complex dimension of $Z$ is greater than $g$, we have $\dim_\bbC(L\cap Z)\ge 1$ (for all components of the intersection). By theorem \ref{coords} we have local holomorphic coordinates $\un{\alpha},\un{\beta},\un{\sigma}$ in a neighborhood of $(\Gamma,p_1,p_2)$ in $\M_{g,2}$, and  $\un{\sigma}$ are local coordinates on the leaf $L$. The function $f_1|_{L\cap Z}$ is the maximum of $\lbrace {\rm Im}\,\phi_i\rbrace$ for all $i$, and thus if $f_1$ achieves a maximum at $(\Gamma,p_1,p_2)\in (L\cap Z)^0$, one of the ${\rm Im}\,\phi_i$ must also achieve a maximum. However, ${\rm Im}\,\phi_i$ is the imaginary part of a holomorphic function on the complex manifold $L$, and thus it is a harmonic function. By the maximum principle, if ${\rm Im}\,\phi_i$ achieves a maximum at an interior point of $(L\cap Z)^0$ (which is not zero-dimensional), then it is constant along $(L\cap Z)^0$. This implies that the value of $f_1$ at all points of $(L\cap Z)^0$ is the same as that at $(\Gamma,p_1,p_2)$, which is the maximum of $f_1$ on $Z$, and thus by definition $(L\cap Z)^0\subset Z_1$.

We now consider the function $f_2|_{Z_1}$, and let $Z_2\subset Z_1$ be the set where it attains its maximum on the compact manifold $Z_1$. We claim that $Z_2$ is still foliated by the leaves of $L\cap Z$, i.e.~that if some $(L\cap Z)^0$ contains $(\Gamma,p_1,p_2)\in Z_2$, then $(L\cap Z)^0\subset Z_2$. To prove this, we use the same argument as above: indeed, $f_2$ is the second-maximum value among the tuple of functions ${\rm Im}\,\phi_i$, and thus if it attains a maximum at some point, one of the functions ${\rm Im}\,\phi_i$ must have a local maximum at this point. Restricting ${\rm Im}\,\phi$ to $(L\cap Z_1)^0$, which we inductively know is equal to the complex manifold $(L\cap Z)^0$, gives a harmonic function, which cannot have a local maximum unless it is constant, and thus the value of $f_2$ is constant along $(L\cap Z_1)^0=(L\cap Z)^0$, so that $(L\cap Z_1)^0\subset Z_2$ by definition.

Repeating this procedure, we get compact real subvarieties $Z_{2g-1}\subset\ldots\subset Z_1\subset Z$ such that for any leaf $L$ containing some $(\Gamma,p_1,p_2)\in Z_{2g-1}$ we still have $(L\cap Z_{2g-1})^0=(L\cap Z)^0$. Let us now consider the local coordinates $\un{\alpha},\un{\beta},\un{\phi}$ near $(\Gamma,p_1,p_2)$. The coordinates $\un{\alpha},\un{\beta}$ are constant on the leaf $L$, while by construction all of $f_i$, and thus all of ${\rm Im}\, \phi_i$, are constant along $Z_{2g-1}$ (achieve their respective maxima everywhere). Since $(L\cap Z)^0$ is a complex variety, if the imaginary part of a (local) holomorphic function on it is constant, the holomorphic function itself is constant. This means that all $\phi_i$, as well as are constant along $(L\cap Z)^0$, and since $\rho,\alpha,\beta$ are by definition constant on the leaf, it means that all the local coordinates given by theorem \ref{coords} are constant along $(L\cap Z)^0$, which implies that $(L\cap Z)^0$ is zero-dimensional.
\end{proof}

Another interesting space to consider is $\M_g^{ct}$, the moduli space of stable curves of compact type, i.e.~those stable curves where the Jacobian is compact; equivalently, this corresponds to pinching a number of separating (homologous to zero) loops on a Riemann surface. In \cite{fp} Faber and Pandharipande further study the vanishing properties of the tautological classes, for $\M_g$ and for the partial compactification $\M_g^{ct}$. They relate the tautological classes on $\overline{\M_g}$ and on the boundary, and use this to prove the vanishing results for tautological rings of both $\M_g$ and $\M_g^{ct}$. In particular their results imply that there do not exist complete subvarieties of $\M_g^{ct}$ of dimension higher than $2g-3$. In fact a stronger result is true:
\begin{prop}[Keel and Sadun \cite{kesa}]
For $g\ge 3$ there do not exist complete complex subvarieties of $\M_g^{ct}$ of dimension greater than $2g-4$.
\end{prop}
\begin{proof}[Idea of the proof from \cite{kesa}]
One uses induction in $g$, the case of $g=3$, when $\M_3^{ct}\to\A_3$ was shown in \cite{kesa} not to contain a threefold, being the base of induction. Suppose  $X\subset\M_g^{ct}$ is a complete subvariety. If $X\subset\M_g$, then by Diaz' theorem its dimension is at most $g-2$, and we are done. Otherwise $X$ must intersect the boundary, and we must have $\dim (X\cap\p\M_g^{ct})=\dim X-1$. Thus there must exist a component $\delta_i=\M_{i,1}^{ct}\times\M_{g-i,1}^{ct}\subset \p\M_g^{ct}$ such that $\dim (X\cap\delta_i)=\dim X-1$. One now uses the inductive bound for the dimension of complete subvarieties of $\M_{i,1}^{ct}$ and $\M_{g-i,1}^{ct}$, and finally observes that since any complete curve in $\M_2^{ct}$ must intersect the boundary, for $g=4$ the hypothetical $X$ would have to intersect not only $\delta_2$, but also $\delta_1$, which gives an improved bound in this case.
\end{proof}

\section{Extension to the boundary}
The line bundle of meromorphic differentials with prescribed pole orders, i.e.~the bundle with fiber $K_\G+\sum c_i p_i$ over a smooth curve, extends to a bundle globally over the Deligne-Mumford compactification $\oM_{g,n}$ --- the fiber over a stable curve $\G$ is $\omega_\G+\sum c_i p_i$, where $\omega_\G$ is the relative dualizing sheaf.
In general if one takes a family of meromorphic differentials on
smooth Riemann surfaces (i.e.~takes a section over $\M_{g,n}$), we
expect that the limit may have simple poles at the nodes. Moreover,
the theory of limit linear series on reducible curves is extremely
complicated, see for example \cite{es-me}, and to determine all possible limits of sections on reducible nodal curves, one may need to twist the bundle by some multiples of the connected components of the nodal curve. We claim that this does not happen for the differentials of the second kind with real periods.
\begin{theo}
The real analytic section $\Psi$ over $\Mh$ of the bundle of meromorphic differentials with one double pole and prescribed singular part extends to a {\it continuous} section of the extension of this bundle, $\omega_\G+2p$ over $\oM_{g,\hat 1}$. For a stable curve $(\G,q,k)$ the section $\Psi_\G$ is the unique meromorphic differential that is identically zero on all connected components of the normalization $\tilde \G$ (geometrically $\tilde \G$ is obtained from $\G$ by detaching the attached nodes) except the one containing $q$. On that component $\Psi_\G$ is the unique differential with real periods and prescribed singular part at the double pole at $q$.
\end{theo}
\begin{proof}
As we did above for differentials of the third kind, choose a point $p_0\ne q$ on $\G$ and consider the function $f(p):={\rm Im}\int_{p_0}^p \Psi$. Since $\Psi$ has real periods, this is a well-defined function on $\G\setminus\lbrace q\rbrace$, and in this case $f$ diverges to both $\pm\infty$ at $q$, i.e.~in any neighborhood of $q$ it takes arbitrarily large and small values. Let us now choose a small open disk $D\subset \G$ around $q$. The function $f$ is a real harmonic function on the open Riemann surface $\G\setminus D$. By the maximum principle it must then achieve its maximum (and also minimum) on the boundary $\partial D$.

Consider now a family $\G_t\subset\oM_{g,\hat 1}$ degenerating to a stable curve $\G_0$. On each $\G_t$ choose a small neighborhood $D_\varepsilon$ of the point $q$, of size $\varepsilon k$, where $k$ is the chosen cotangent vector at $q$, not containing any nodes for any $t$. This is always possible --- if the family degenerates by acquiring nodes away from $q$, this is clear; for the degeneration when a node develops and approaches $q$ the stable model has a blowup at this point, and thus on the blown up $\P^1$ there is an open disk around the marked point there not containing the nodes.

Then the function $f_t$ is bounded on $X_t\setminus D_\varepsilon$ above and below by its values on $\partial D_\varepsilon$. However, since the singular part of $\Psi$ at $q$ is prescribed, we can write down the singular part of the expansion of $f_t$ near $q$, and thus the values of $f_t$ on $\partial D_\varepsilon$ are bounded independent of $t$. Thus the limit function $f(p)={\rm Im}\int^p\Psi_0$ must also be bounded on $X_0\setminus D_\varepsilon$. We want to show that $\Psi_0$ is equal to the differential $\Phi$ determined by the condition of its holomorphicity at the nodes. Indeed, let us take $F(p):={\rm Im}\int^p(\Psi_0-\Phi)$. This is a real harmonic function on $X_0$, bounded on $X_0\setminus D_\varepsilon$ by the above argument, but also bounded in the neighborhood $D_\varepsilon\supset p_0$ since $\Psi_0-\Phi$ is holomorphic at $p_0$. Thus $F$ is a bounded harmonic function on a compact Riemann surface $X_0$ (to be more precise, on each component of the normalization), and thus is constant, which implies $\Psi_0=\Phi$.
\end{proof}
\begin{rem}
The compactification of $\oM_{g,1}$ has boundary strata corresponding to the case of the marked point approaching the node in the limit --- in this case the stable reduction is to take a blowup, and thus we would end up with an attached $\bbC\bbP^1$ with a marked point and fixed coordinate $z$ at the marked point, which can be extended to $\bbC\bbP^1$. In this case $dz/z^2$ is the unique meromorphic differential on $\bbC\bbP^1$ with a double pole and given singular part --- there are no periods to consider.
\end{rem}
\begin{rem}
In the theory of limit linear series determining the limit of a line bundle on a reducible curve is very complicated \cite{es-me}, and in fact the compactification of the universal Picard scheme over $\oM_g$ has several connected components over the boundary \cite{ca}. In dealing with limit linear series, it may not even be enough to consider limits in the versal deformation space \cite{mu}: studying the limit of the line bundle for degenerating families of curves with higher order tangency to the boundary of $\oM_g$ may be needed. Note, however, that in our proof we only use the maximum principle for harmonic functions, which holds independent of the degenerating family considered.
\end{rem}
\begin{rem}
The above theorem fails for differentials of the third kind, i.e.~for sections $\Psi_{X,p,q,k}$ over $\M_{g,2}$. Indeed, if one tries to apply the same proof, the neighborhoods of both points $p$ and $q$ need to be removed, and $f$ could achieve its maximum on the boundary of one neighborhood, and the minimum on the boundary of the other. If points $p$ and $q$ lie on different components of the nodal curve $X_0$, then on any component we would only have either the lower or the upper bound for $f_t$, and thus it is possible for the limit $\Psi_0$ to acquire simple poles with residues $\pm k$ at the nodes of $X_0$. It can in fact be shown that this is the only possible limit, i.e.~that no twisting of the bundle by the components of the reducible curve is possible.
\end{rem}

\end{document}